\documentclass[reqno]{amsart}

\usepackage{graphicx}
\usepackage{amssymb}
\usepackage[all]{xy}
\usepackage[inline]{enumitem}

 \newtheorem{theorem}{Theorem}[section]
 \newtheorem{proposition}[theorem]{Proposition}

 \newcommand{\rank}{\mathop{\rm rank}\nolimits}

 \newcommand{\cat}{\mathop{\rm cat}\nolimits}
 \newcommand{\nil}{\mathop{\rm nil}\nolimits}
 \newcommand{\TC}{\mathop{\rm TC}\nolimits}

 \newcommand{\Ker}{\mathop{\rm Ker}\nolimits}

 \newcommand{\ZZ}{{\mathbb{Z}}}
 \newcommand{\QQ}{{\mathbb{Q}}}
 
 \newcommand{\CC}{{\mathbb{C}}}
 \newcommand{\HH}{{\mathbb{H}}}

	\newcommand{\FF}{{\mathbb{F}}}

\begin{document}

\title{Manifolds with small topological complexity}

\author[Petar Pave\v{s}i\'{c}]{Petar Pave\v si\'{c}$^{*}$}
\address{$^{*}$ Faculty of Mathematics and Physics, University of Ljubljana,
Jadranska 21,  1000 Ljubljana, Slovenija}
\email{petar.pavesic@fmf.uni-lj.si}
\thanks{$^{*}$ Supported by the Slovenian Research Agency program P1-0292 and grants 
N1-0083, N1-0064}

\date{\today}

\begin{abstract}
We study closed orientable manifolds whose topological complexity is at most 3 and determine their cohomology
rings. For some of admissible cohomology rings we are also able to identify corresponding manifolds up to
homeomorphism. 
\ \\[3mm]
{\it Keywords}: topological complexity, Lusternik-Schnirelmann category, closed manifold, 
zero divisor cup-length \\
{\it AMS classification: 55R70, 55M30} 
\end{abstract}

\maketitle


\section{Introduction}

Topological complexity of a space $X$, denoted $\TC(X)$ is a numerical homotopy invariant 
introduced by M. Farber \cite{Farber2003} as a quantitative measure for the complexity of motion planning in 
a configuration space $X$ of some robot device.
Although configuration spaces of robots can be quite general topological spaces
(cf. Kapovich-Millson \cite{KapovichMillson2002}, Pave\v si\'c \cite{Pavesic2018}),
of particular importance are those that have the structure of a closed manifold (e.g., ordered
configuration spaces of manifolds, see Cohen \cite{Cohen2018}, configuration spaces of spidery
linkages, see O'Hara \cite{OHara2007} and of general parallel mechanisms, see 
Shvalb-Shoham-Blanc \cite{ShvalbShohamBlanc2005}).
It is thus of interest to determine which closed manifolds $M$ have a given 
value of $\TC(M)$.
The case $\TC(M)=1$ is void, because a non-trivial closed manifold cannot be
contractible. Grant, Lupton and Oprea \cite[Corollary 1.2]{GrantLuptonOprea2013} showed that the only
closed manifolds with topological complexity equal to 2 are the odd-dimensional spheres. 
In this paper we will make a further step and study closed oriented manifolds $M$ with $\TC(M)=3$. Some
examples immediately spring to mind: even-dimensional spheres $S^{2n}$ by \cite[Theorem 8]{Farber2003}
and products of two odd-dimensional spheres, by \cite[Theorem 8 and Theorem 11]{Farber2003}. 
The question is, are there any other examples? We will give an answer in the main result of the next section,
Theorem \ref{thm:coho structure} in which we exactly describe admissible cohomology rings of manifolds 
whose topological complexity is at most 3. In Section 3 we discuss explicit manifolds whose cohomology 
ring is described in the mentioned theorem. 

For a topological space $X$ let $X^I$ denote the space of continuous paths 
$\alpha\colon I\to X$, and let $\pi\colon X^I\to X\times X$ be the evaluation map
$\pi(\alpha):=\big(\alpha(0),\alpha(1)\big)$. \emph{Topological complexity} of a path-connected
topological space $X$ is the least integer $\TC(X)=n$ for which there exists 
a covering $U_1,\ldots,U_n$ of $X\times X$, where each $U_i$ is open and admits a continuous 
section to the map $\pi\colon X^I\to X\times X$ (cf. \cite[Definition 2]{Farber2003}). Note that 
if $X$ is a compact ANR space (which includes closed manifolds) then the requirement that the sets in the covering are open is superfluous, since by \cite[Theorem 4.6]{Pavesic2019},
one can consider coverings of $X\times X$ by arbitrary subsets. Note also, that many authors 
(e.g., the above mentioned article \cite{GrantLuptonOprea2013}) use a normalized
or reduced topological complexity, which is one less that in our definition, 
so that the topological complexity of a contractible space is 0 instead of 1. 

The main properties of topological complexity are listed in the following proposition, where the value of 
$\TC(X)$ is related to the Lusternik-Schnirelmann category $\cat(X)$ (for which we refer to the classical
monograph \cite{CLOT}), and to the nilpotency of certain ideal in the cohomology ring of $X\times X$. 

\begin{proposition}\ \\[-6mm]
\label{prop:TC properties}
\begin{enumerate}
\item $\TC(X)=1$ if, and only if $X$ is contractible;
\item Homotopy invariance: $$X\simeq Y \Rightarrow \TC(X)=\TC(Y);$$
\item Category estimate: $$\cat(X)\le \TC(X)\le \cat(X\times X);$$
\item If $X$ is a topological group, then $\TC(X)=\cat(X)$;
\item Cohomological estimate: $$\TC(X)\ge \nil(\Ker\Delta^*),$$ where 
$\Delta^*\colon H^*(X\times X;R)\to H^*(X;R)$ 
is the homomorphism induced by the diagonal map
$\Delta\colon X\to X\times X$ on the cohomology with coefficients in a ring $R$, and $\nil(\Ker\Delta^*)$ is the minimal integer $k$ for which all
$k$-fold products in $\Ker\Delta^*$ are zero;
\item Product formula: $$\TC(X\times Y)\le \TC(X)+\TC(Y)-1.$$
\end{enumerate}
\end{proposition}

Recall that the value of $\Delta^*$ on the cross-product $u\times v\in H^*(X\times X;R)$ of elements 
$u,v\in H^*(X;R)$ can be given in terms of their cup-product as 
$$\Delta^*(u\times v)=u\cdot v,$$
and the cup-product of elements $u\times v$ and $u'\times v'$ is given as
$$ (u\times v)\cdot(u'\times v')=(-1)^{|v|\cdot |u'|} (u\cdot u')\times (v\cdot v'),$$
where $|v|,|u'|$ are the dimensions of cohomology classes $v$ and $u'$ (see \cite[pp. 215-216]{Hatcher2002}). 
This explains why Farber \cite[Definition 6]{Farber2003} called $\Ker \Delta^*$ the 
\emph{ideal of zero-divisors} of $H^*(X;R)$.
For every $u\in H^*(X;R)$ we have 
$$\Delta^*(u\times 1-1\times u)=u\cdot 1-1\cdot u=0,$$ 
therefore $(u\times 1-1\times u)\in\Ker \Delta^*$. Indeed, if $H^*(X;R)$ is a free $R$-module (which implies 
that 
$H^*(X\times X;R)\cong H^*(X;R)\otimes H^*(X;R)$ by Kunneth theorem), then one can easily show that 
$\Ker \Delta^*$ is generated as an ideal by elements of the form $(u\times 1-1\times u)$.

\section{Admissible cohomology rings}

Computation of topological complexity of closed surfaces was completed in the orientable case by Farber 
\cite[Theorem 9]{Farber2003} and in the non-orientable case by Dranishnikov \cite{Dranishnikov2017}
and Cohen and Vandembroucq \cite{CohenVandembroucq2017}. By mentioned results the only closed surfaces whose
topological complexity is 3 are the sphere $S^2$ and the torus $S^1\times S^1$. 
To avoid making unnecessary exceptions, from this point on let $M$ denote a closed, orientable $m$-dimensional
 manifold with $m\ge 3$. 

We will base our computations on the following corollary of a theorem proved by Dranishnikov, Katz 
and Rudyak \cite{DranishnikovKatzRudyak2008}. 

\begin{proposition}
If $\TC(M)\le 3$, then $\pi_1(M)$ is free. 
\end{proposition}
\begin{proof}
If $\TC(M)\le 3$, then $\cat(M)\le 3$ by Proposition \ref{prop:TC properties}(3), and then 
the claim follows immediately from \cite[Theorem 1.1]{DranishnikovKatzRudyak2008}.
\end{proof}

We will proceed by considering several cases and sub-cases. Let $g$ denote the 
generator of $H^m(M;R)$ and for every $u\in H^*(M;R)$ let
$$\widehat u:=u\times 1-1\times u \in H^*(M\times M;R)$$ 
be the shorthand for the corresponding zero-divisor.

(1) Let us first assume that the rank of $\pi_1(M)$ is at least 2 and consider the cup-product pairing 
$$ H^1(M;\ZZ)\times H^{m-1}(M;\ZZ)\stackrel{\cdot}{\longrightarrow} H^m(M;\ZZ)$$
which is non-singular by \cite[Proposition 3.38]{Hatcher2002}. Since the rank of $H^1(M;\ZZ)$
is at least 2, non-singularity of the pairing implies that there exist linearly
independent elements $u,v\in H^1(M;\ZZ)$ and $u',v'\in H^{m-1}(M;\ZZ)$ such that
$u\cdot u'=v\cdot v'=g$, and 
furthermore $u\cdot v'=v\cdot u'=0$. Then we have the following non-trivial four-fold product
$$\widehat{u}\cdot\widehat{u'}\cdot\widehat{v}\cdot\widehat{v'}=2 (g\times g)\ne 0.$$
Therefore, if $\rank(\pi_1(M))\ge 2$, then $\TC(M)\ge 5$ by Proposition \ref{prop:TC properties}(6). 

(2) If $\pi_1(M)\cong\ZZ$ let $u$ be a generator of $H^1(M;\ZZ)\cong\ZZ$ and let, as before, 
$v\in H^{m-1}(M;\ZZ)$ be such that $u\cdot v=g$. If $m-1$ is even, then we have
$$\widehat{v}^2\cdot\widehat{u}=-2(v\times v)\cdot\widehat{u}=-2(g\times u-u\times g)\ne 0,$$
(note that $v^2=0$ for dimensional reasons) and thus $\TC(M)\ge 4$ by Proposition 
\ref{prop:TC properties}(6). On the other hand, if $m-1$ is odd, and if there exists 
a non-zero element $w\in H^i(M;\ZZ)$
for some $2\le i\le m-2$, then 
$$\widehat{u}\cdot\widehat{v}\cdot\widehat{w}=w\times g-g\times w\pm uw\times v-v\times uw\ne 0, $$
so again $\TC(M)\ge 4$. 

We conclude that if $\pi_1(M)\cong\ZZ$ and $\TC(M)=3$, then $H^*(M;\ZZ)$ is generated by 
two cohomology classes in dimensions 1 and $m-1$, which are Poincar\'e duals to each other, and furthermore 
$m-1$ must be odd. In other words, $H^*(M;\ZZ)\cong \bigwedge(x_1,x_k)$ for some odd integer $k>1$.

(3) If $M$ is simply-connected, then we consider four sub-cases according to the structure of 
the group $$\widehat{H}(M;R):=\bigoplus_{i=2}^{m-2} H_i(M;R).$$

(3a) If $\widehat H(M;\QQ)\ne 0$ we argue similarly as in case (2). First of all we note that 
$H^*(M;\ZZ)$ is not all torsion, so by \cite[Corollary 3.39]{Hatcher2002} we may find elements 
$u,v\in H^*(M;\ZZ)$ of infinite order, such that $u\cdot v=g$. As in case (2), if 
either $u$ or $v$ is of even degree, then we can find a non-trivial product of three zero-divisors, 
and then $\TC(M)\ge 4$. Therefore, if $\TC(M)\le 3$, 
then  both $u$ and $v$ must be of odd degree, which as before implies that $H^*(M;\ZZ)$ contains 
a subring 
of the form $\bigwedge(x_k,x_l)$ where $k,l$ are odd integers and $1<k\le l< m-1$. 
Furthermore, if there exists an element $w\in H^*(M;\ZZ)$ which is not contained in the 
mentioned subring, then $\widehat{u}\cdot\widehat{v}\cdot\widehat{w}\ne 0$ similarly as
in the second part of case (2). Thus, $\widehat H(M;\QQ)\ne 0$ and $\TC(M)=3$ imply 
$H^*(M;\ZZ)\cong \bigwedge(x_k,x_l)$.

(3b) Let us now assume that $\widehat H(M;\QQ)=0$ but $\widehat H(M;\FF_p)\ne 0$ for some 
odd prime $p$, and let $k$ be the minimal $k\ge 2$ for which $H_k(M;\ZZ)$ has $p$-torsion. 
By the universal coefficient theorem for cohomology (see \cite[Theorem 3.2]{Hatcher2002})
$H^i(M;\FF_p)\ne 0$ for $i=k,k+1$. It then follows by Poincar\'e duality that $H^i(M;\FF_p)\ne 0$ for 
$i=m-k-1,m-k$. Therefore, $H^i(M;\FF_p)\ne 0$ in three different dimensions, unless 
$m=2k+1$. In the first case, we may find (as in case (2)) three non-trivial cohomology classes $u,v,w$ 
of different dimension (with $u\cdot v=g$ by \cite[Corollary 3.39]{Hatcher2002}), for which 
$\widehat{u}\cdot\widehat{v}\cdot\widehat{w}\ne 0$ and thus $\TC(M)\ge 4$. 

On the other hand, if $m=2k+1$, then let $u\in H^k(M;\FF_p)$ and $v\in H^{k+1}(M;\FF_p)$ be
such that $u\cdot v=g$. If $k$ is even, then 
$$\widehat{u}^2\cdot\widehat{v}=2(u\times g-g\times u)+v\times u^2-u^2\times v\ne 0.$$
Similarly, if $k$ is odd, then $\widehat{u}\cdot\widehat{v}^2\ne 0$, so in both cases $\TC(M)\ge 4$.

(3c) The next sub-case arises if $\widehat H(M;\QQ)=0$ and $\widehat H(M;\FF_p)= 0$ for $p$ odd but  
$\widehat H(M;\FF_2)\ne 0$. The argument is similar as in (3b), except if $m=2k+1$, since in that
case the proof that $\widehat{u}^2\cdot\widehat{v}\ne 0$ for $k$ even (or that 
$\widehat{u}\cdot\widehat{v}^2\ne 0$ for $k$ odd) breaks down because of 2-torsion. 
On the other hand, if $u\in H^k(M;\FF_2)$ and $v\in H^{k+1}(M;\FF_2)$ 
such that $u\cdot v=g$, and if additionally $u^2\ne 0$, then 
$$\widehat{u}^2\cdot\widehat{v}=u^2\times v+v\times u^2\ne 0,$$
which implies $\TC(M)\ge 4$. Therefore, under the assumptions of (3c), if $\TC(M)=3$ then
$H^*(M;\FF_2)\cong \bigwedge(x_k,x_{k+1})\otimes \FF_2$.

(3d) The final possibility is that $\widehat{H}(M;R)=0$ for all coefficient rings $R$,
which clearly implies that $H^*(M;\ZZ)\cong \bigwedge(x_k)$.      

We may summarize the above discussion in the following theorem:

\begin{theorem}
\label{thm:coho structure}
Assume that $M$ is a closed, orientable manifold with $\TC(M)\le 3$. Then $\pi_1(M)$ is a free group 
and one of the following alternatives holds:\\
(1) $H^*(M;\ZZ)\cong \bigwedge(x_k)$, $1\le k$, or \\
(2) $H^*(M;\ZZ)\cong \bigwedge(x_k,x_l)$, $k,l$ odd, $1\le k\le l$, or\\
(3) $H_i(M;\ZZ)=0$ for $i\ne 0,k,m$ and $H^*(M;\FF_2)\cong \bigwedge(x_k,x_{k+1})\otimes \FF_2$.
\end{theorem}

\section{Some manifolds with small TC}

Theorem \ref{thm:coho structure} clearly shows that the condition $\TC(M)\le 3$ is much more restrictive than 
the analogous condition $\cat(M)\le 3$. Indeed the class of manifolds whose Lusternik-Schnirelmann category 
is at most 3 includes all surfaces, two-fold products of sphere, all $(n-1)$-connected $2n$-manifolds
and a variety of other examples. In this section we will try to collect some information about the
actual manifolds $M$ satisfying $\TC(M)\le 3$. For some cohomology rings we will be able to determine 
exactly the corresponding manifolds, while in other cases we will only present suitable candidates and compute
their Lusternik-Schnirelmann category.

{\bf 1.} The simplest case to consider are manifolds whose cohomology ring is given by Theorem \ref{thm:coho structure}(1).
In fact it is straightforward that $H^*(M;\ZZ)\cong \bigwedge(x_k)$ implies that $M$ is homotopy
equivalent to $S^k$. The positive solution to the Poincar\'e conjecture then implies that $M$ is actually
homeomorphic to $S^k$.
\\[3mm] \
{\bf 2.} If $H^*(M;\ZZ)\cong \bigwedge(x_1,x_k)$ as in Theorem \ref{thm:coho structure}(2), then (since 
$S^1\simeq K(\ZZ,1)$) there is a map $f_1\colon M\to S^1$ which represents 
$$x_1\in H^1(M;\ZZ)\cong [M,S^1].$$
Similarly, there exists a map $f_k\colon M\to K(\ZZ,k)$ representing the cohomology class
$$x_k\in H^k(M;\ZZ)\cong [M,K(\ZZ,k)].$$
It is well-known that $K(\ZZ,k)$ can be obtained by attaching  cells of dimension bigger or equal 
to $k+2$ to the sphere $S^k$. Since the dimension of $M$ is $m=k+1$, we may assume by cellular approximation
 theorem
that the image of $f_k$ is contained in $S^k$, and so we have a map $f\colon X\to S^k$. It is easy to check
that the resulting map 
$$(f_1,f_k)\colon M\to S^1\times S^k$$
is an isomorphism on the integral cohomology and thus a homotopy equivalence, since $\pi_1(M)\cong\ZZ$.  
Then a result of Kreck and L\"uck \cite[Theorem 0.13(a)]{KreckLuck2008} implies that $M$ is actually
homeomorphic to $S^1\times S^k$.
\\[3mm] \
{\bf 3.} If $H^*(M;\ZZ)\cong \bigwedge(x_k,x_k)$ with $k$ odd, then $M$ is a $(k-1)$-connected 
$2k$-dimensional manifold
and one has C.T.C. Wall's classification \cite{Wall1962} by which $M\approx S^k\times S^k$ provided 
$k\cong 3,5,7 (\mathrm{mod}\  8)$ (cf. also \cite[Theorem 3.1]{Bokoretal2019}).
\\[3mm] \
{\bf 4.} The instances of  Theorem \ref{thm:coho structure}(2) when $H^*(M;\ZZ)\cong \bigwedge(x_k,x_l)$ for 
$1< k < l$ with $k,l$ odd are more complicated. The products of odd spheres of the form $S^k\times S^l$ 
have the abovementioned cohomology ring and $\TC(S^k\times S^l)=3$. Note that by the above-mentioned result of 
Kreck and L\"uck, a manifold that is homotopy equivalent to a product of odd spheres is actually 
homeomorphic to that product. Another example is 
the special unitary group $SU(3)$ whose cohomology  is $H^*(SU(3);\ZZ)\cong \bigwedge(x_3,x_5)$.
Indeed, Singhof \cite[Theorem 1(a)]{Singhof1975} proved that $\cat(SU(3))=3$, therefore by Proposition 
\ref{prop:TC properties}(4), we have $\TC(SU(3))=3$, as well.

Another potential candidate is the symplectic group $Sp(2)$ whose cohomology is 
$H^*(Sp(2);\ZZ)\cong \bigwedge(x_3,x_7)$. However, Schweitzer \cite{Schweitzer1965} used secondary cohomology operations to prove 
that $\cat(Sp(2))=4$, which in turn implies that $\TC(Sp(2))=4$. Hilton and Roitberg 
\cite{HiltonRoitberg1969} discovered
three more examples of H-spaces whose cohomology is isomorphic to $\bigwedge(x_3,x_7)$, which are usually 
denoted  $E_{3\omega}, E_{4\omega}, E_{5\omega}$. 
Their Lusternik-Schnirelmann category (and thus topological complexity) is equal to 4 (see 
\cite[Chapter 4]{CLOT}).

More generally, let us consider fibre bundles $p\colon M\to S^l$ with fibre $S^k$ for 
odd integers $1<k< l$. The cohomology of $M$ is easily computed using Gysin sequence, so we obtain 
$H^*(M;\ZZ)\cong \bigwedge(x_k,x_l)$ and the manifold itself admits a CW-decomposition of the form
$$M=S^k\cup_\alpha e^l\cup_\beta  e^{k+l},$$
with attaching maps $\alpha\colon S^{l-1}\to S^k$ and $\beta\colon S^{k+l-1}\to S^k\cup_\alpha e^l$.
If $\alpha$ is a suspension (e.g, if $l<2k-1$ so that $\pi_{l-1}(S^k)$ is in the stable range), then 
$\cat(S^k\cup_\alpha e^l)=2$ and therefore $\cat(M)=3$. This includes important special cases as the
complex and quaternionic Stiefel manifolds, $V_2(\CC^n)=U(n)/U(n-2)$ whose cohomology ring is 
$H^*(V_2(\CC^n);\ZZ)\cong\bigwedge(x_{2n-1},x_{2n-3})$, and 
$V_2(\HH^n)=Sp(n)/Sp(n-2)$ with $H^*(V_2(\HH^n);\ZZ)\cong\bigwedge(x_{4n-1},x_{2n-5})$. 

If the attaching map $\alpha$ is not a suspension, then $\cat(S^k\cup_\alpha e^l)=3$. In that case 
$\cat(M)=3$ if, and only if certain set of Hopf invariants 
$\mathcal{H}(\beta)$ contains the zero class (see Chapter 6 of \cite{CLOT}, in particular Theorem 6.19). 
As we see, there are many sphere bundles over spheres, whose category is 3. Unfortunately, we 
are currently lacking a general method to determine their topological 
complexity, so this remains an interesting open problem. One possible approach is to apply
certain higher  Hopf invariants, a method that was recently developed by Gonzalez, Grant and 
Vandembroucq \cite{GonzalezGrantVandembroucq2019}. 
They managed to compute topological complexity of many two-cell complexes, but the technical details are quite
formidable, and the full analysis of three-cell complexes is probably very hard. 
Nevertheless, we were able to combine some of their computations with our results from \cite{Pavesic2020}
that relate topological complexity of a space with topological complexity of its skeleta, to show that 
some sphere bundles over spheres have 
topological complexity at least 4. We will work in the so-called \emph{meta-stable range} and assume that
 $2k<l<3k-1$. Under this assumption one can associate to every map $\alpha\colon S^{l-1}\to S^k$  
a \emph{generalized Hopf invariant} 
$H_0(\alpha)\colon S^{l-1}\to S^{2k-1}$ (see \cite[Section 5]{GonzalezGrantVandembroucq2019} for relevant
definitions and results), which allow to determine $\TC(S^k\cup_\alpha e^l)\ge 4$. 

\begin{proposition}
Let $k$ be an odd integer and let $2k<l<3k-1$. Assume that $M$ has a CW-decomposition of the form 
$M=S^k\cup_\alpha e^l\cup_\beta  e^{k+l}$ with attaching maps $\alpha\colon S^{l-1}\to S^k$ and 
$\beta\colon S^{k+l-1}\to S^k\cup_\alpha e^l$ (this in particular applies if $M$ is an $S^l$-bundle
over $S^k$). 
If $H_0(\alpha)\ne 0$, then $\TC(M)\ge 4$.
\end{proposition}
\begin{proof} Note that the inclusion $S^k\cup_\alpha e^l\hookrightarrow M$ is a $(k+l-1)$-equivalence
because $S^k\cup_\alpha e^l$ is the $(k+l-1)$-skeleton of $M$. 
Topological complexity of $S^k\cup_\alpha e^l$ was bounded from below in 
\cite[Theorem 5.6]{GonzalezGrantVandembroucq2019}: $\TC(S^k\cup_\alpha e^l)\ge 4$. 
On the other hand, \cite[Theorem 3.6]{Pavesic2020} implies that $\cat(M)\ge \cat(S^k\cup_\alpha e^l)=3$,
therefore $\TC(M)\ge 3$. Then we may apply
\cite[Theorem 3.1]{Pavesic2020}, which states that if 
$$2\dim(S^k\cup_\alpha e^l)< k\,(\TC(M)-1)+(k+l-1)$$
(which is clearly satisfied if $l<3k-1$),
then $\TC(M)\ge \TC(S^k\cup_\alpha e^l)\ge 4$.
\end{proof}

Observe that the first 
'undecided' cases arise in dimension 10, which is of interest if one considers
configuration spaces of specific mechanical systems. 

From a different perspective, one may also consider the Morse decomposition of a manifold 
$M$ with $H^*(M;\ZZ)\cong \bigwedge(x_k,x_l)$
as above. Smale \cite[Theorem G]{Smale1961} showed that if dimension of $M$ is at least 6, than it has a Morse
decomposition with the minimal number of handles compatible with its homology. Therefore, 
$M$ admits a decomposition with four handles whose indices are $0,k,l$ and $k+l$ respectively. 
The union of the $0$- and $k$-handles depends on the framing which is given by an element of $\pi_{k-1}(O(l))$.
This group is known to be trivial for $k\not\equiv 1(\!\!\!\mod 8)$, therefore the union of the first two 
handles is homeomorphic to $S^k\times D^l$. By the same argument the union of the $l$- and $(k+l)$-handles
is also homeomorphic to $S^k\times D^l$. We may conclude that under these assumptions
($\TC(M)=3$, $H^*(M;\ZZ)\cong \bigwedge(x_k,x_l)$ with $k\not\equiv 1(\!\!\!\mod 8)$) the manifold
$M$ can be obtained by glueing together two copies of $S^k\times D^l$ along the common 
boundary $S^k\times S^l$.
\\[3mm] \
{\bf 5.} Let us finally consider manifolds that satisfy condition (3) of Theorem \ref{thm:coho structure}. 
The lowest 
dimensional case is a simply-connected 5-dimensional manifold whose $\FF_2$ cohomology is 
$H^*(M;\FF_2)\cong \bigwedge(x_2,x_3)\otimes\FF_2$. Barden \cite{Barden1965} showed that every simply-connected
5-dimensional manifolds can be decomposed as a connected sum of certain basic 5-manifolds. We are not 
dwelling into details but one can easily check that the only 5-manifold that satisfies the above condition
is the famous Wu manifold $SU(3)/SO(3)$. It admits a CW-decomposition $SU(3)/SO(3)=S^2\cup e^3\cup e^5$, where 
the 3-cell is attached by a degree 2 map, therefore the 3-skeleton of $SU(3)/SO(3)$ is the Moore space
$M(\ZZ/2,2)$. The category of a Moore space is 2, therefore the category of the Wu manifold is 3. 
However, we were not able to check whether its topological complexity is also 3. One can construct higher
analogues of the Wu manifold using handle decompositions, for example by gluing together two copies of a 
(twisted of untwisted, depending on the dimension) $D^{k+1}$-bundle over $S^k$ along a suitable 
homeomophism of the boundary. All of these spaces have a CW-decomposition with the top-cell attached to 
a suspension, so their category is equal to 3. 

\

We should also mention an interesting result that was recently proved by S. Mescher 
\cite[Proposition 6.2]{Mescher2020}. He used weighted 
cohomology classes to show that a closed oriented manifold $M$ with $\TC(M)\le 3$ is either a rational
homology sphere or it admits a degree 1 map from a closed oriented manifold of the form $S^1\times P$
(i.e., it is \emph{dominated} by some product of a manifold of dimension $\dim(M)-1$ with a circle). 

\

Let us conclude with a brief discussion on two possible extensions of the presented results. Theorem 
\ref{thm:coho structure} gives a precise description of cohomology rings of closed orientable 
manifolds whose topological complexity is at most 3, so it is natural to ask what can be said 
about non-orientable closed manifolds $M$ with $\TC(M)\le 3$. As in the orientable case, the fundamental group
$\pi_1(M)$ must be free. That rank of $\pi_1(M)$ cannot exceed 1 can be seen similarly as in Section 2. 
On the other hand, $\pi_1(M)$ cannot be trivial, because $M$ is non-orientable. We thus conclude 
that $H^*(M;\FF_2)\cong \bigwedge(x_1,x_{m-1})\otimes \FF_2$, and the corresponding manifolds are the
generalized Klein-bottles (non-orientable $S^{m-1}$-bundles over $S^1$). Their category is 3 but we 
do not know  whether their topological complexity can be, at least in some cases, also equal to 3.

Another extension that could be pursued is determination of manifolds whose topological complexity is at 
most 4. Although the general case seems to be beyond reach because we have very little information on 
manifolds whose category is $4$, we believe that some reasonable progress could be achieved
on closed manifolds $M$ satisfying $\TC(M)\le 4$ and $\cat(M)\le 3$.

\end{document}